\documentclass{amsart}
\usepackage{graphicx}
\newtheorem{thm}{Theorem}[section]

\newtheorem{lem}[thm]{Lemma}

\theoremstyle{definition}

\theoremstyle{remark}
\newtheorem{rem}[thm]{Remark}

\numberwithin{equation}{section}

\begin{document}

\title[Losses in $GI^X/GI^Y/1/n$ queues]
{Characterization theorem on losses in $GI^X/GI^Y/1/n$ queues}%
\author{{\bf Vyacheslav M. Abramov},\\
\textit{Swinburne University of Technology}
}%
\footnote{Vyacheslav M. Abramov, Center for Advanced Internet
Architectures, Faculty of Information and Communication
Technologies, Swinburne University of Technology, John Street, PO
Box 218, Hawthorn, Victoria 3122, Australia, email:
vabramov126@gmail.com}
%
%
\begin{abstract}
In this paper, we prove a characterization theorem on the number of
losses during a busy period in $GI^X/GI^Y/1/n$ queueing systems, in
which the interarrival time distribution belongs to the class NWUE.
\end{abstract}
\maketitle
%
%

\section{Introduction}
\noindent There are several papers that study the properties of
losses from queues during their busy periods. In \cite{Abramov 1997}
and then in \cite{Righter 1999} and \cite{Wolff 2002}, it was proved
that in $M/GI/1/n$ queueing systems, in which the expectations of
interarrival and service times are equal, the expected number of
losses during their busy periods is equal to 1 for all $n$. In
\cite{Abramov 2001b}, this result was extended for $M^X/GI/1/n$
queues. In \cite{Abramov 2001a} and \cite{Abramov 2006} some
stochastic inequalities connecting the number of losses during busy
periods in $M/GI/1/n$ and $GI/M/1/n$ queues and the number of
offspring in Galton-Watson branching processes were obtained, and in
\cite{Wolff 2002}, for $GI/GI/1/n$ queueing systems in which
interarrival time distribution belongs to the class NBUE or NWUE,
simple inequalities for the expected number of losses during a busy
period were obtained. Pek\"oz, Righter and Xia \cite{Pekoz Righter
Xia 2003} gave a characterization of the number of losses during a
busy period of $GI/M/1/n$ queueing systems. They proved that if the
expected number of losses during a busy period is equal to 1 for all
$n$, then arrivals must be Poisson.

In the present paper, we prove a characterization theorem for the
expected number of losses during a busy period for the class of
$GI^X/GI^Y/1/n$ queueing systems, in which interarrival time
distribution belongs to the class NWUE.

Recall that the probability distribution function of a random
variable $\xi$ is said to belong to the class NBUE if for any
$x\geq0$ the inequality $\mathsf{E}\{\xi-x|\xi>x\}\leq\mathsf{E}\xi$
holds. If the opposite inequality holds, i.e.
$\mathsf{E}\{\xi-x|\xi>x\}\geq\mathsf{E}\xi$, then the probability
distribution function of a random variable $\xi$ is said to belong
to the class NWUE.

The queueing system $GI^X/GI^Y/1/n$, where the symbols $X$ and $Y$
denote an arrival batch and, respectively, service batch, is
characterized by parameter $n$ and four (control) sequences
$\{\tau_i, X_i, \chi_i, Y_i\}$ of random variables ($i=1,2,\ldots$),
each of which consists of independently and identically distributed
random variables, and these sequences are independent of each other
(e.g. see \cite{Borovkov 1976}). Let $A(x)=\mathsf{P}\{\tau_i\leq
x\}$ denote the interarrival time probability distribution function,
$a=\int_0^\infty x\mathrm{d}A(x)<\infty$, and let
$B(x)=\mathsf{P}\{\chi_i\leq x\}$ denote the service time
probability distribution function, $b=\int_0^\infty
x\mathrm{d}B(x)<\infty$. The random variables $X_1$, $X_2$,\ldots
denote consecutive masses of arriving units or, in known
terminology, their batch sizes; but here they are assumed to be
positive real-valued random variables rather than integer-valued. In
turn, the random variables $Y_1$, $Y_2$,\ldots denote consecutive
service masses or service batches, which are also assumed to be
positive real-valued random variables. The $i$th service batch $Y_i$
characterizes the quantity that can be processed during the $i$th
service time given that the necessary quantity is available in the
system immediately before the $i$th service. Both $X_1$ and $Y_1$
are assumed to have finite expectations. In addition, the capacity
of the system $n$ is assumed to be a positive real number in
general.

The queueing systems with real-valued batch arrival and service, as
well as with real-valued capacity are not traditional. They can be
motivated, however, in industrial applications, where units can be
lorries with sand or soil arriving in and departing from a storage
station. In usual queueing formulations, where the random variables
$X_i$ and $Y_i$ are integer-valued, they are characterized as
batches of arrived and served customers. The main result of the
present paper, Theorem \ref{thm3}, is new even in the particular
case when $X_i=Y_i=1$ and $n$ is integer.

\begin{thm}\label{thm3}  Let $M_L$
 denote the total mass lost during a busy period.
Assume that the probability distribution function $A(x)$ belongs to
the class NWUE. Then the equality $\mathsf{E}M_L=\mathsf{E}X_1$
holds for all positive real $n$ if and only if arrivals are Poisson,
the random variable $Y_1$ takes a single value $d$, the probability
distribution function of $X_1$ is lattice with span $d$, and
$\mathsf{E}X_1=\frac{ad}{b}$.
\end{thm}

This theorem is true for full and partial rejection policies,
work-conserving disciplines and can be adapted to different models
considered, for instance, in \cite{Wolff 2002}. The characterization
theorem is a necessary and sufficient condition that includes the
case of the Poisson arrivals. However, if arrivals are not Poisson
but belong to the practically important class NWUE, then in the case
where $ \frac{1}{a}\mathsf{E}X_1\geq\frac{1}{b}\mathsf{E}Y_1$, i.e.
in the case where the total mass of arrivals per unit time is not
smaller than the total mass of service per unit time, we have the
simple inequality given by Lemma \ref{thm2} that is used to prove
Theorem \ref{thm3}.

\section{Proof of the main result}
\noindent

\noindent We start from the following lemma.

\begin{lem}\label{thm2}
Assume that the length of a busy period has a finite mean, and that
the probability distribution function $A(x)$ belongs to the class
NWUE. If $ \frac{1}{a}\mathsf{E}X_1\geq\frac{1}{b}\mathsf{E}Y_1$ and
$\mathsf{P}\{X_1\leq n\}>0$, then for any nontrivial random variable
$Y_1$ (i.e. taking at least two positive values) we have
$\mathsf{E}M_L>\mathsf{E}X_1$.
\end{lem}

\begin{proof}
Let $N_A$ denote the total number of arrivals during a busy cycle
(that is, total number of arrivals during a busy period plus the
unit that starts the busy period), and let $N_S$ denote the total
number of service completions during the busy period. Denoting by
$M_A$ the total mass of arrivals during a busy cycle, and,
respectively, by $M_S$ the total mass of served units during a busy
period. Using Wald's identity, we have:
\begin{eqnarray}
\mathsf{E}M_A&=&\mathsf{E}X_1\mathsf{E}N_A ,\label{eq1}\\
a\mathsf{E}N_A&=&b\mathsf{E}N_S+\mathsf{E}I,\label{eq2}
\end{eqnarray}
where $I$ in Equation \eqref{eq2} denotes the length of idle time.
Since $A(x)$ belongs to the class NWUE, then $\mathsf{E}I\geq a$
(see \cite{Wolff book}, p.482). Hence, from \eqref{eq2} we have
\begin{equation}\label{eq3}
a\mathsf{E}N_A-a\geq b\mathsf{E}N_S.
\end{equation}

Note, that for $M_S$ we cannot use Wald's identity directly in order
to show that $\mathsf{E}M_S\leq\mathsf{E}Y_1\mathsf{E}N_S$.
In order to prove this inequality, we introduce the sequence of
random variables ${S_1}$, ${S_2}$,\ldots that characterizes
\textit{real} masses of service or, in other words, \textit{real}
batch sizes of service satisfying the properties
$\mathsf{E}\{{S_j}|N_S=j\}\leq\mathsf{E}Y_1$ while
$\mathsf{E}\{{S_i}|N_S=j\}=\mathsf{E}Y_1$ for $i<j$ ($1\leq i<j$).
Apparently ${S_1}$, ${S_2}$,\ldots are \textit{not} independent
random variables. So, additional properties of the sequence ${S_1}$,
${S_2}$,\ldots are needed in order to establish the inequality for
$\mathsf{E}M_S$.

Let $m_i$ denote the workload of the system immediately before the
service of the $i$th unit starts. Then, given $\{m_1=x_1$, $m_2=x_2,
\ldots\}$ $(x_i\leq n$ for all $i$), the sequence ${S_1}$,
${S_2}$,\ldots is conditionally independent. Hence, under the
condition $\{m_1=x_1$, $m_2=x_2, \ldots\}$ for the sequence of
conditionally independent random variables ${S_1}$, ${S_2}$,\ldots
one can use the following theorem by Kolmogorov and Prohorov
\cite{Kolmogorov and Prohorov}.

\begin{lem}\label{lem1} (Kolmogorov and Prohorov \cite{Kolmogorov and Prohorov}.)
Let $\xi_1$, $\xi_2$,\ldots be independent random variables, and let
$\nu$ be an integer random variable such that the event $\{\nu=k\}$
is independent of $\xi_{k+1}$, $\xi_{k+2}$,\ldots. Assume that
$\mathsf{E}\xi_k=v_k$, $\mathsf{E}|\xi_k|=u_k$ and the series
$\sum_{k=1}^\infty\mathsf{P}\{\nu\geq k\}u_k$ converges. Then,
$$
\mathsf{E}\sum_{i=1}^\nu\xi_i=\sum_{k=1}^\infty\mathsf{P}\{\nu=k\}\sum_{i=1}^k
v_i.
$$
\end{lem}

Note, that the condition $\sum_{k=1}^\infty\mathsf{P}\{\nu\geq
k\}u_k<\infty$ of Lemma \ref{lem1} is satisfied, because
$\mathsf{E}N_S<\infty$ and $\mathsf{E}S_k\leq n$ for all $k$. Hence,
by the total expectation formula we have $\mathsf{E}{S_j}\leq
\mathsf{E}Y_1$, and consequently by Lemma \ref{lem1} and the total
expectation formula we arrive at
$\mathsf{E}M_S\leq\mathsf{E}N_S\mathsf{E}Y_1$. We show below, that
in fact we have the strict inequality
$\mathsf{E}M_S<\mathsf{E}N_S\mathsf{E}Y_1$.

 Indeed, the fact that the probability distribution function
$A(x)$ belongs to the class NWUE implies that
$\overline{A}(x)=1-A(x)>0$ for any $x$. Hence, taking into account
that $Y_1$ takes at least two different positive values, one can
conclude that there exists the value $j_0$ such that
$\mathsf{E}\{S_{j_0}|N_S=j_0\}<\mathsf{E}Y_1$, and consequently
$\mathsf{E}{S_{j_0}}<\mathsf{E}Y_1$. This implies
\begin{equation}
\label{eq4}\mathsf{E}M_S<\mathsf{E}N_S\mathsf{E}Y_1.
\end{equation}
Now \eqref{eq1}, \eqref{eq3} and \eqref{eq4} and the equality
$\mathsf{E}M_L=\mathsf{E}M_A-\mathsf{E}M_S$ allows us to obtain the
inequality $\mathsf{E}M_L>\mathsf{E}X_1$.
\end{proof}

\begin{rem}
In the formulation of the Lemma \ref{thm2} we assumed that the
length of a busy period has a finite mean. This assumption is
technically important in order to use Wald's identity. Note, that
the assumption that the interarrival time distribution belongs to
the class NWUE implies $\overline{A}(x)=1-A(x)>0$ for any $x$, which
consequently enables us to conclude that a busy period always exists
(i.e. finite) with probability 1. If the expectation of the busy
period length is infinite, then the expected number of losses during
a busy period is infinite as well, and hence the statement of Lemma
\ref{thm2} in this case remains true.
\end{rem}

\textit{Proof of Theorem \ref{thm3}}. Note first that if $
\frac{1}{a}\mathsf{E}X_1<\frac{1}{b}\mathsf{E}Y_1$, then
$\mathsf{E}M_L$ vanishes as $n\to\infty$ due to the law of large
numbers.
 Hence, the
only case $ \frac{1}{a}\mathsf{E}X_1\geq\frac{1}{b}\mathsf{E}Y_1$ is
available, and this is the assumption in Lemma \ref{thm2}.

Hence, the problem reduces to a minimization problem for
$\mathsf{E}M_L$ in the set of the possible values. More
specifically, the problem is to find the infimum of $\mathsf{E}M_L$
subject to the constraints given by \eqref{eq1}, \eqref{eq3},
\eqref{eq4} and the inequality $
\frac{1}{a}\mathsf{E}X_1\geq\frac{1}{b}\mathsf{E}Y_1$. Then, the
statement of this theorem follows if and only if along with
\eqref{eq1} we also have $a\mathsf{E}N_A-a= b\mathsf{E}N_S$,
$\mathsf{E}M_S=\mathsf{E}N_S\mathsf{E}Y_1$ and $
\frac{1}{a}\mathsf{E}X_1=\frac{1}{b}\mathsf{E}Y_1$. The first
equality follows if and only if arrivals are Poisson, and the second
one follows if and only if $Y_1$ takes a single value $d$, and the
probability distribution function of $X_1$ is lattice with span $d$.
Then, the system of three equations together with the equality
$\mathsf{E}X_1=\frac{ad}{b}$ (which in turn is a consequence of $
\frac{1}{a}\mathsf{E}X_1=\frac{1}{b}\mathsf{E}Y_1$) yields the
desired result $\mathsf{E}M_L=\mathsf{E}X_1$.

\section*{Acknowledgement}
\noindent This paper was written when the author was with the
Electronic Engineering Department in City University of Hong Kong.
The support and hospitality of Professor Moshe Zukerman are highly
appreciated.

\end{document}